\documentclass[12pt,reqno]{amsart}

\newcommand\version{November 22, 2009}

\usepackage{amsmath,amsfonts,amsthm,amssymb,amsxtra}

\newtheorem{theorem}{Theorem}[section]

\newtheorem{lemma}[theorem]{Lemma}
\newtheorem{corollary}[theorem]{Corollary}

\theoremstyle{definition}

\newtheorem{example}[theorem]{Example}
\newtheorem{remark}[theorem]{Remark}

\numberwithin{equation}{section}

\newcommand{\C}{\mathbb{C}}

\newcommand{\const}{\mathrm{const}\ }

\renewcommand{\epsilon}{\varepsilon}

\renewcommand{\phi}{\varphi}
\newcommand{\R}{\mathbb{R}}

\newcommand{\Sph}{\mathbb{S}}

\DeclareMathOperator{\supp}{supp}

\newcommand{\lanbox}{\hfill \hbox{$\, \vrule height 0.25cm width 0.25cm depth 0.01cm \,$}}

\begin{document}

\title[Inversion positivity --- \version]{Inversion positivity and \\ the sharp Hardy--Littlewood--Sobolev inequality}

\author{Rupert L. Frank}
\address{Rupert L. Frank, Department of Mathematics,
Princeton University, Washington Road, Princeton, NJ 08544, USA}
\email{rlfrank@math.princeton.edu}

\author{Elliott H. Lieb}
\address{Elliott H. Lieb, Departments of Mathematics and Physics,
Princeton University,
       P.~O.~Box 708, Princeton, NJ 08544, USA}
\email{lieb@princeton.edu}

\thanks{\copyright\, 2009 by the authors. This paper may be reproduced, in its entirety, for non-commercial purposes.\\
Support through DFG grant FR 2664/1-1 (R.F.) and U.S. NSF grant PHY 0652854 (R.F. and E.L.) is gratefully acknowledged.}

\begin{abstract}
We give a new proof of certain cases of the sharp HLS inequality. Instead of symmetric decreasing rearrangement it uses the reflection positivity of inversions in spheres. In doing this we extend a characterization of the minimizing functions due to Li and Zhu.
\end{abstract}

\maketitle

%%%%%%%%%%%%%%%%%%%%%%%%%%%%%%%%%%%%%%%%%%%%%%%%%%%%%%%%%%%%%%%%%%%%%%%%%%%%%%%%%%%%%%%%%%%

\section{Introduction and main result}

The Hardy--Littlewood--Sobolev inequality for functions on $\R^N$,
\begin{equation} \label{eq:hls1}
\Big| \, I_\lambda[f,g] \, \Big|
 \leq \mathcal{H}_{N,\lambda, p,q} \|f\|_p\|g\|_q \,,
\end{equation}
where
$$
I_\lambda[f,g] := \iint_{\R^N\times\R^N} \frac{\overline{f(x)}\ g(y)}{|x-y|^\lambda} \,dx\,dy
$$
holds for all $0<\lambda <N$ and $p,q>1$ with $1/p+1/q +\lambda/N =2$. It is important in several areas of
analysis and it is interesting to find the sharp constant $\mathcal{H}_{N,\lambda, p,q } $ 
whenever possible. Of particular interest is the \emph{diagonal case} $p=q = 2N/(2N-\lambda)$,
where the best choice for $g$ is $g=f$ (because $|x-y|^{-\lambda}$ is positive definite). In this case the sharp constants were found in \cite{Li} by recognizing that stereographic 
projection from $\R^N$ to the sphere $\Sph^N$ turns the maximizing $f$ into the constant
function on $\Sph^N$ .  This is the only case for which the sharp constants are known, although bounds exist for $p\neq q$. A simplification of the proof was then made by Carlen and Loss 
\cite{CaLo} using the method of `competing symmetries', which they invented. In both proofs
a major input was the Riesz rearrangement inequality, which allowed one to 
restrict attention to symmetric decreasing functions $f$. A discussion of these proofs is in
\cite[Sec. 4.3 and 4.6]{LiLo}. 

Among the diagonal cases, an important example is $\lambda= N-2$,
where the kernel is Newton's gravitation potential. Mathematically,
this case is dual to the ordinary Sobolev inequality for $N\geq
3$, \cite[Thm. 8.3]{LiLo} $\Vert \nabla f\Vert^2_2 \geq S_N \Vert f
\Vert^2_{2N/(N-2)}$, and thus the sharp constant for one gives a sharp
constant for the other.  Completely different proofs have been given
for this special case \cite{Au,Ta,CoNaVi,BoLe}. Similarly, $\lambda=N-1$
corresponds to the Sobolev inequality for $\sqrt{-\Delta}$ when $N\geq 2$.

In this paper we give a new proof of the diagonal case of \eqref{eq:hls1}
which does not use symmetric decreasing rearrangements. If $N\geq 3$
the additional assumption $\lambda\geq N-2$ is required, but this
covers the most important cases in applications. Our proof is based
on the conformal invariance of the problem \cite[Sec. 4.4]{LiLo} and
reflection positivity of the left side of (\ref{eq:hls1}) with respect
to inversions in certain spheres, together with an interesting geometric
idea of Li and Zhu \cite{LZh}.  The concept of reflection positivity
through planes \cite{OsSc,Li0,FrIsLiSi,GlJa0,GlJa} has a long history
and more recently Lopes and Mari\c{s} \cite{LoMa} used it effectively
to prove spherical symmetry of certain functional minimizers. Our main
contribution is reflection positivity via inversions in spheres instead
of reflections in planes and we hope that this concept will also be
useful elsewhere. The genesis of this idea was the use of moving spheres
instead of moving planes in \cite{LZh} and \cite{L} and their geometric
characterization of the optimizers in \eqref{eq:hls1}. The motivation
in \cite{LZh} and \cite{L} was to replace moving planes by moving spheres, 
while the motivation here is to replace reflection positivity through 
planes by reflection positivity through spheres.

We go a bit beyond \cite{LZh}, however, by extending their analysis from
continuous functions to finite Borel measures on $\R^N$. We prove that
the only measures that are invariant with respect to these particular
conformal transformations must be absolutely continuous with respect
to Lebesgue measure and their densities must be the well known functions
$f(x) =  \alpha \left(\beta +|x-y|^2\right)^{-(2N-\lambda)/2}$.

A precise statement of the theorem we will prove is the following.

\begin{theorem}[HLS inequality]\label{mainhls}
 Let $0<\lambda<N$ if $N=1,2$ and $N-2\leq\lambda<N$ if $N\geq 3$. If $p=q=2N/(2N-\lambda)$, then \eqref{eq:hls1} holds with
\begin{equation}
 \label{eq:hlsconst}
\mathcal H_{N,\lambda,p,p} = \pi^{\lambda/2} \frac{\Gamma((N-\lambda)/2)}{\Gamma(N-\lambda/2)} \left(\frac{\Gamma(N)}{\Gamma(N/2)}\right)^{1-\lambda/N} \,.
\end{equation}
Equality holds if and only if
$$
f(x) = \alpha \left(\beta +|x-y|^2\right)^{-(2N-\lambda)/2}
\quad\text{and}\quad
g(x) = \alpha' \left(\beta +|x-y|^2\right)^{-(2N-\lambda)/2} \,,
$$
for some $\alpha,\alpha'\in\C$, $\beta>0$ and $y\in\R^N$.
\end{theorem}

%%%%%%%%%%%%%%%%%%%%%%%%%%%%%%%%%%%%%%%%%%%%%%%%%%%%%%%%%%%%%%%%

\subsection*{Outline of the proof of Theorem \ref{mainhls}}

As observed in \cite{Li, CaLo}, $I_\lambda$ is conformally invariant. We shall use the fact that the value of $I_\lambda[f]:=I_\lambda[f,f]$ does not change if $f$ is inverted on the surface of a ball or reflected on a hyperplane. To state this precisely, we need to introduce some notation.

Let $B=\{x \in\R^N : \ |x-a|<r\}$, $a\in\R^N$, $r>0$, be an open ball and denote by
$$
\Theta_B(x) := \frac{r^2(x-a)}{|x-a|^2} +a
$$
the inversion of a point $x\neq a$ through the boundary of $B$. This map on $\R^N$ can be lifted to an operator acting on functions $f$ on $\R^N$ according to
$$
(\Theta_B f)(x) := \left(\frac{r}{|x-a|}\right)^{2N-\lambda} f(\Theta_B(x)) \,.
$$
(Strictly speaking, $\Theta_B f$ is not defined at the point $x=a$.) Note that both the map and the operator $\Theta_B$ satisfy $\Theta_B^2=I$, the identity. By the change of variables $z=\frac{r^2(x-a)}{|x-a|^2} +a$ and using $dz= \left(\frac{r}{|x-a|}\right)^{2N} dx$ and
$$
\left| \frac {r^2(x-a)}{|x-a|^2} -\frac {r^2(y-a)}{|y-a|^2} \right| = \frac{r}{|x-a|}\ |x-y| \ \frac{r}{|y-a|} \,,
$$
one easily finds that
\begin{equation}
 \label{eq:confinv}
I_\lambda[f]=I_\lambda[\Theta_B f] \,.
\end{equation}
Similarly, let $H=\{x \in\R^N : \ x\cdot e>t \}$, $e\in\Sph^{N-1}$, $t\in\R$, be a half-space and denote by
$$
\Theta_H(x) := x+ 2(t-x\cdot e)
$$
the reflection of a point $x$ on the boundary of $H$. The corresponding operator is defined by
$$
(\Theta_H f)(x) := f(\Theta_H(x))
$$
and it again satisfies $\Theta_H^2=I$ and
\begin{equation}
 \label{eq:confinvh}
I_\lambda[f]=I_\lambda[\Theta_H f] \,.
\end{equation}
Our first ingredient in the proof of Theorem \ref{mainhls} is the following.

\begin{theorem}[Reflection and inversion positivity]\label{infrefpos}
 Let  $0<\lambda<N$ if $N=1,2$, $N-2\leq\lambda<N$ if $N\geq 3$, and let $B\subset\R^N$ be either a ball or a half-space. If $f\in L^{2N/(2N-\lambda)}(\R^N)$ and
\begin{equation*}
 f^i (x) :=
\begin{cases}
 f(x) & \text{if}\ x\in B\,,\\
\Theta_B f(x) & \text{if}\ x\in \R^N\setminus B\,,
\end{cases}
\qquad
f^o (x) :=
\begin{cases}
 \Theta_B f(x) & \text{if}\ x\in B \,, \\
f(x) & \text{if}\ x\in \R^N\setminus B \,,
\end{cases}
\end{equation*}
then
\begin{equation}\label{eq:infrefpos}
\frac12 \left(I_\lambda[f^i]+I_\lambda[f^o]\right) \geq I_\lambda[f] \,.
\end{equation}
If $\lambda>N-2$ then the inequality is strict unless $f=\Theta_B f$.
\end{theorem}

For half-spaces and $\lambda=N-2$ (the Newtonian case) this theorem was long known to quantum field theorists \cite{OsSc,Li0,GlJa0,GlJa}. In this case, reflection positivity is equivalent to the assertion that, as operators, the Dirichlet Laplacian is bigger than the Neumann Laplacian on the half-space. The half-space case with $N-2<\lambda<N$ (but not the strictness for $\lambda>N-2$) was apparently first proved by Lopes and Mari\c{s} \cite{LoMa}. The case of balls seems to be new for all $\lambda$.

\begin{remark}
 The restriction $\lambda\geq N-2$ for $N\geq 3$ is necessary for \eqref{eq:infrefpos} to hold. Indeed, for $0<\lambda<N-2$ the quantity $\frac12 \left(I_\lambda[f^i]+I_\lambda[f^o]\right) -I_\lambda[f]$ can attain both positive and negative values for $f\in L^{2N/(2N-\lambda)}(\R^N)$, see Remark \ref{counter}. Moreover, for $\lambda=N-2$ it can vanish without having $f=\Theta_B f$, see Example \ref{countercou}.
\end{remark}

Our second main ingredient is a generalization of Li and Zhu's theorem \cite{LZh}; see also \cite{L}.

\begin{theorem}[Characterization of inversion invariant measures]\label{yyl}
 Let $\mu$ be a finite, non-negative measure on $\R^N$. Assume that 
\begin{itemize}
 \item[(A)] \label{ass:yyl}
for any $a\in\R^N$ there is an open ball $B$ centered at $a$ and for any $e\in\Sph^{N-1}$ there is an open half-space $H$ with interior unit normal $e$ such that
\begin{equation}
 \label{eq:yyl}
\mu(\Theta^{-1}_B(A))=\mu(\Theta^{-1}_H(A))=\mu(A)
\qquad \text{for any Borel set}\ A\subset\R^N \,.
\end{equation}
\end{itemize}
Then $\mu$ is absolutely continuous with respect to Lebesgue measure and
$$
d\mu(x) = \alpha \left(\beta +|x-y|^2\right)^{-N} dx
$$
for some $\alpha\geq 0$, $\beta>0$ and $y\in\R^N$.
\end{theorem}

We emphasize that $B$ and $H$ in assumption (A) divide $\mu$ in half, in the sense that $\mu(B)=\mu(\R^N\setminus\overline B)$ and $\mu(H)=\mu(\R^N\setminus\overline H)$. Moreover, by a change of variables one finds that for absolutely continuous measures $d\mu=v \,dx$ assumption (A) is equivalent to the fact that for any $a\in\R^N$ there is an $r_a>0$ and a set of full measure in $\R^N$ such that for any $x$ in this set
\begin{equation}\label{eq:pweq}
v(x) = \left(\frac{r_a}{|x-a|}\right)^{2N} v\left( \frac{r_a^2 (x-a)}{|x-a|^2}+a \right) \,,
\end{equation}
and similarly for reflections.

\begin{remark}
 The assumption that $\mu$ is finite is essential, since $d\mu(x)=|x|^{-2N}dx$ also satisfies assumption (A).
\end{remark}

We now show how Theorem \ref{mainhls} follows from Theorems \ref{infrefpos} and \ref{yyl}.

\begin{proof}
 The kernel
$$
|x-y|^{-\lambda} = \const \int_{\R^N} |x-z|^{-(\lambda+N)/2} |y-z|^{-(\lambda+N)/2} \,dz \,, 
$$
is positive definite and therefore we take $g=f$ henceforth. Let $f$ be an optimizer, that is, a non-trivial function $f\in L^p(\R^N)$, $p=2N/(2N-\lambda)$, for which the supremum
$$
\mathcal H_{N,\lambda,p,p} = \sup\left\{ \frac{I_\lambda[h]}{\|h\|_p^2} :\ 0\not\equiv h \in L^p(\R^N) \right\}
$$
is attained. The existence of such a function was shown, e.g., in \cite{Li}. (An alternative proof in \cite{Ln} does not use the technique of symmetric decreasing rearrangements.) Of course, we may assume that $f\geq 0$. For any point $a$ there is a ball $B$ centered at $a$ such that $\int_{B} f^p \,dx = \int_{\R^N\setminus B} f^p \,dx$. We note that if $f^i$ and $f^o$ are defined as in Theorem~\ref{infrefpos} then $\|f^i\|_p=\|f^o\|_p=\|f\|_p$. Moreover, by \eqref{eq:infrefpos}, $\tfrac12( I_\lambda[f^i]+I_\lambda[f^o] )\geq I_\lambda[f]$ and hence, in particular, $\max\{I_\lambda[f^i], I_\lambda[f^o]\} \geq I_\lambda[f]$. By the maximizing property of $f$ this inequality cannot be strict, and therefore we conclude that $I_\lambda[f^i]= I_\lambda[f^o]= I_\lambda[f]$, that is, both $f^i$ and $f^o$ are optimizers as well.

 In order to continue the argument we assume first that either $N=1,2$ or else that $N\geq 3$ and $\lambda>N-2$. Since we have just shown that one has equality in \eqref{eq:infrefpos}, the second part of Theorem \ref{infrefpos} implies that $f=\Theta_{B} f$. By a similar argument one deduces that $f=\Theta_H f$ for any half-space such that $\int_H f^p \,dx = \int_{\R^N\setminus H} f^p \,dx$. Therefore the measure $f^p\,dx$ satisfies the assumption of Theorem \ref{yyl}, and hence $f$ has the form claimed in Theorem \ref{mainhls}. The value of the $\mathcal H_{N,\lambda,p,p}$ is found by explicit calculation, e.g., via stereographic projection; see \cite{Li}.

Now assume that $N\geq 3$ and $\lambda=N-2$. The difference from the previous case is that there is no strictness assertion in Theorem \ref{infrefpos} (indeed, equality in \eqref{eq:infrefpos} can hold without $f=\Theta_B f$), so we need an additional argument to conclude that $f=\Theta_B f$ for any ball and half-space with $\int_B f^p \,dx = \int_{\R^N \setminus B} f^p \,dx$. This argument is in the spirit of \cite{Lo,LoMa}. We have already proved that $f^o$ (and $f^i$) are optimizers. The corresponding Euler-Lagrange equations are
$$
\int_{\R^N} \frac{f(y)}{|x-y|^{N-2}} \,dy = \mu \ f^{p-1}(x) \,,
\qquad
\int_{\R^N} \frac{f^o(y)}{|x-y|^{N-2}} \,dy = \mu \ (f^o)^{p-1}(x) \,,
$$
where the Lagrange multipliers coincide since $I_{N-2}[f]=I_{N-2}[f^o]$ and $\|f^o\|_p=\|f\|_p$. Define $u:=f^{p-1}$ and $u^o := (f^o)^{p-1}$. Then
$$
-\Delta u = \tilde\mu \ u^{p'-1} \,,
\qquad
-\Delta u^o = \tilde\mu \ (u^o)^{p'-1}
$$
where $\tilde\mu:=\mu^{-1} (N-2)|\Sph^{N-1}|$. The function $w:=u-u^o$ satisfies $-\Delta w+Vw=0$ with
$$
V(x):= -\tilde \mu \frac{u^{p'-1}(x) - (u^o)^{p'-1}(x)}{u(x)-u^o(x)} 
= -\tilde \mu (p'-1) \int_0^1 \left(t u(x)+ (1-t)u^o(x)\right)^{p'-2} \,dt \,.
$$
Note that $w\equiv 0$ in $\R^N\setminus B$. Using the unique continuation theorem from \cite{JeKe} we are going to deduce that $w\equiv 0$ everywhere, and hence $f=f^o$. In order to verify the assumptions of \cite{JeKe} we note that $u=f^{p-1}\in L^{p'}(\R^N)$ and similarly for $u^o$. From this one easily deduces that $V\in L^{N/2}(\R^N)$. Moreover, $-\Delta w= \tilde\mu \left( f-f^o\right) \in L^{p}(\R^N)$. Under these conditions the argument in \cite{JeKe} implies that $w\equiv 0$. Hence $f=\Theta_Bf$ and we can deduce Theorem \ref{mainhls} again from Theorem \ref{yyl}.
\end{proof}

\subsection*{Acknowledgement}
We are grateful to E. Carlen for pointing out that the conformal invariance of the HLS functional and the conventional reflection positivity through planes imply the inversion positivity through spheres. This allows us to circumvent our original, direct but complicated proof, which uses properties of Gegenbauer polynomials.

%%%%%%%%%%%%%%%%%%%%%%%%%%%%%%%%%%%%%%%%%%%%%%%%%%%%%%%%%%%%%%%%

\section{Reflection and inversion positivity}\label{sec:pos}

Our goal in this section is to prove Theorem \ref{infrefpos}. In Subsection \ref{sec:repr} we consider the case of half-spaces and we shall derive a representation formula for $I_\lambda[\Theta_H f, f]$. In Subsection \ref{sec:cayley} we show how the case of balls can be reduced to the case of half-spaces, and in Subsection \ref{sec:infrefposproof} we give the proof of Theorem \ref{infrefpos}.

%%%%%%%%%%%%%%%%%%%%%%%%%%%%%%%%%%%%%%%%%%%%%%%%%%%%%%%%%%%%%%%%%

\subsection{Reflection positivity}
\label{sec:repr}

Throughout this subsection we assume that $H=\{x\in\R^N:\ x_N> 0 \}$. The key for proving Theorem \ref{infrefpos} is the following explicit formula for $I_\lambda[\Theta_H f,f]$.

\begin{lemma}[Representation formula]
 \label{repr}
 Let $0<\lambda<N$ if $N=1,2$ and $N-2\leq \lambda<N$ if $N\geq 3$. Let $f\in L^{2N/(2N-\lambda)}(\R^N)$ be a function with support in $\overline H = \{x\in\R^N:\ x_N\geq 0 \}$. If $\lambda>N-2$, then
\begin{equation}
 \label{eq:repr}
I_\lambda[\Theta_H f,f] = c_{N,\lambda} 
\int_{\R^{N-1}} d\xi'
\int_{|\xi'|}^\infty d\tau \frac{\tau^2}{(\tau^2-|\xi'|^2)^{(N-\lambda)/2}}
\left| \int_\R \frac{\hat f(\xi)} {\tau^2+\xi_N^2} \,d\xi_N \right|^2
\end{equation}
where
$$
c_{N,\lambda} =  2^{N+1-\lambda} \pi^{(N-4)/2} \, \frac{\sin(\pi(N-\lambda)/2) \ \Gamma((N-\lambda)/2)}{\Gamma(\lambda/2)} >0 \,.
$$
If $\lambda=N-2$, then
\begin{equation}
 \label{eq:reprcou}
I_{N-2}[\Theta_H f,f] = \frac{4 \pi^{(N-2)/2}}{\Gamma((N-2)/2)} \ 
\int_{\R^{N-1}} d\xi' |\xi'|
\left| \int_\R \frac{\hat f(\xi)} {|\xi'|^2+\xi_N^2} \,d\xi_N \right|^2 \,.
\end{equation}
\end{lemma}

When $N=1$, we use the convention that $\R^{N-1}=\{0\}$ and that $d\xi'$ gives measure $1$ to this point. Note that
by the invariance \eqref{eq:confinvh} the left side of \eqref{eq:repr} is real-valued for any (possibly complex-valued) $f$.

The crucial point of Lemma \ref{repr} is, of course, that the right sides of \eqref{eq:repr} and \eqref{eq:reprcou} are non-negative. This is no longer the case for $0<\lambda<N-2$ if $N\geq 3$, see Remark~\ref{counter} below.

Formula \eqref{eq:reprcou} and its proof are well-known and our proof of \eqref{eq:repr} follows the same strategy. An essentially equivalent form of \eqref{eq:repr} has recently appeared in \cite{LoMa} with a different proof. 

\begin{proof}
 If $N=1$ we have
$$
I_\lambda[\Theta_H f,f] = \int_0^\infty \int_0^\infty \frac{\overline{f(x)}\ f(y)}{(x+y)^{\lambda}} \,dx\,dy 
= \frac1{\Gamma(\lambda)} \int_0^\infty \frac{d\tau}{\tau^{1-\lambda}} \left| \int_0^\infty e^{-\tau x} f(x) \,dx \right|^2 \,.
$$
Recalling that $f(x)=0$ for $x\leq 0$ and using that $e^{-\tau|\cdot|}$ has Fourier transform $(2/\pi)^{1/2} \tau/(\xi^2+\tau^2)$ we can write
\begin{equation}
 \label{eq:cauchy}
\int_0^\infty e^{-\tau x} f(x) \,dx = \sqrt\frac2\pi \ \tau \int_\R \frac{\hat f(\xi)}{\xi^2+\tau^2} \,d\xi \,.
\end{equation}
Noting that $c_{1,\lambda}= 2/(\pi \Gamma(\lambda))$ we arrive at the assertion for $N=1$.

For $N\geq 2$ the functional is
$$
I_\lambda[\Theta_H f,f] = \int_H \int_H \frac{\overline{f(x)}\ f(y)}{(|x'-y'|^2+(x_N+y_N)^2)^{\lambda/2}} \,dx\,dy \,.
$$
Using the Fourier transform of $|x|^{-\lambda}$ (see, e.g., \cite[Thm. 5.9]{LiLo} where, however, another normalization is used) we can rewrite this as
\begin{align*}
 I_\lambda[\Theta_H f,f] & = (2\pi)^{-N+1} \tilde c_{N,\lambda} \int_{\R^N}  
\int_H \int_H \overline{f(x)} \ \frac{e^{i\xi'\cdot(x'-y') + i\xi_N(x_N+y_N)}}{|\xi|^{N-\lambda}} f(y) \,dx\,dy\,d\xi \\
& = \tilde c_{N,\lambda} \int_{\R^{N-1}} J_{\lambda,\xi'}[F_{\xi'}] \,d\xi' \,,
\end{align*}
where $\tilde c_{N,\lambda}=2^{N-1-\lambda} \pi^{(N-2)/2} \Gamma((N-\lambda)/2)/ \Gamma(\lambda/2)$,
$$
F_{\xi'}(t) := (2\pi)^{-(N-1)/2} \int_{\R^{N-1}} f(x',t) e^{-i\xi'\cdot x'} \,dx' \,,
$$
and
$$
J_{\lambda,\xi'}[\phi] = \!\int_0^\infty \!\!\int_0^\infty \!\overline{\phi(t)} k_{\lambda,\xi'}(t+s) \phi(s) \,ds\,dt \,,
\quad
k_{\lambda,\xi'}(t) := \!\int_\R \frac{e^{i\xi_N t} }{(|\xi'|^2+\xi_N^2)^{(N-\lambda)/2}} \,d\xi_N \,.
$$
Note that for $\xi'\neq0$, $k_{\lambda,\xi'}$ converges absolutely if $N-2\leq \lambda< N-1$ and as an improper Riemann integral (that is, $\lim_{R\to\infty} \int_{-R}^R$ ) if $N-1\leq\lambda<N$.

Using complex analysis we shall write $k_{\lambda,\xi'}$ as the Laplace transform of a positive measure. First, assume that $N\geq 3$ and $\lambda=N-2$. Then by the residue theorem
$$
k_{\lambda,\xi'}(t) = \pi |\xi'|^{-1} e^{-t|\xi'|} \,,
$$
and hence
$$
J_{N-2,\xi'}[\phi] = \pi |\xi'|^{-1} \left| \int_0^\infty \! e^{-t|\xi'|} \phi(t) \,dt \right|^2 \,.
$$
In view of \eqref{eq:cauchy} this is the claimed formula. Now let $N\geq 2$ and $N-2<\lambda<N$. We observe that for fixed $t$ and $\xi'$, the function $e^{i\xi_N t} (|\xi'|^2+\xi_N^2)^{-(N-\lambda)/2}$ of $\xi_N$ is analytic in the upper halfplane with the cut $\{ i\tau :\ \tau\geq |\xi'| \}$ removed. Deforming the contour of integration to this cut and calculating the jump of the argument along it we obtain
$$
k_{\lambda,\xi'}(t)
= \int_\R \frac{e^{i\xi_N t} }{(|\xi'|^2+\xi_N^2)^{(N-\lambda)/2}} \,d\xi_N
= 2 \sin\left(\tfrac\pi 2 (N-\lambda)\right) \int_{|\xi'|}^\infty \frac{e^{-\tau  t}}{(\tau^2-|\xi'|^2)^{(N-\lambda)/2}} \,d\tau \,.
$$
Hence we find
$$
J_{\lambda,\xi'}[\phi] = 2 \sin\left(\tfrac\pi 2 (N-\lambda)\right)
\int_{|\xi'|}^\infty \frac{d\tau}{(\tau^2-|\xi'|^2)^{(N-\lambda)/2}}
\left| \int_0^\infty \! e^{-\tau t} \phi(t) \,dt \right|^2 \,.
$$
Using again \eqref{eq:cauchy} we obtain the assertion.
\end{proof}

\begin{remark}
 \label{finiteenrem}
Lemma \ref{repr} remains valid for $f\in\dot H^{-(N-\lambda)/2}(\R^N)$. More precisely, for $f\in L^{2N/(2N-\lambda)}(\R^N)$ one has $\hat f\in L^{2N/\lambda}(\R^N)$ by Hausdorff-Young and
\begin{equation}\label{eq:finiteenrem}
I_\lambda[f] = a_{\lambda,N} \int_{\R^N} |\xi|^{-N+\lambda} |\hat f(\xi)|^2 \,d\xi
\end{equation}
with some constant $a_{\lambda,N}>0$. This is finite as long as $|\xi|^{-(N-\lambda)/2} \hat f\in L^2(\R^N)$, i.e., $f\in\dot H^{-(N-\lambda)/2}(\R^N)$. Note that such $f$ could be distributions that are not functions. For $f\in\dot H^{-(N-\lambda)/2}(\R^N)$, $\Theta_H f$ can be defined by duality, and \eqref{eq:confinvh} remains valid. The quantity $I[\Theta_H f,f]$ is defined by \eqref{eq:finiteenrem} and polarization. Since $I_\lambda[\Theta_H f,f] \leq I_\lambda[f]^{1/2} I_\lambda[\Theta_H f]^{1/2} = I_\lambda[f]$, formulas \eqref{eq:repr} and \eqref{eq:reprcou} extend by continuity to all $f\in\dot H^{-(N-\lambda)/2}(\R^N)$.
\end{remark}

As we have already pointed out, what is crucial for us is that the right sides of \eqref{eq:repr} and \eqref{eq:reprcou} are non-negative. Indeed, in Subsection \ref{sec:infrefposproof} we shall see that the right side of \eqref{eq:repr} is strictly positive unless $f\equiv 0$. This is not true for \eqref{eq:reprcou}, as the following counterexample shows.

\begin{example}
 \label{countercou}
Let $N\geq 3$. Let $f\in L^{2N/(N+2)}(\R^N)$ be radially symmetric around a point $a\in\R^N$ with $a_N>0$, let $f$ have support in $H=\{x:\ x_N>0\}$ and assume that $\int_{\R^N} f(x)\,dx = 0$. Then by Newton's theorem 
$$
\int_{\R^N} |x-y|^{-N+2} f(y) \,dy = 0
\qquad\text{if $x$ is outside the convex hull of}\ \supp f \,.
$$
In particular, the integral vanishes for $x\in\supp\Theta_H f$ and therefore $I_{N-2}[\Theta_H f,f]=0$.
\end{example}

\begin{remark}
 \label{counter}
Let $N\geq 3$ and $0<\lambda<N-2$. We claim that $I_\lambda[\Theta_H f,f]$ assumes both positive and negative values for functions $f\in L^{2N/(2N-\lambda)}(\R^N)$ with support in $\overline H$. Indeed, one still has
\begin{align*}
 I_\lambda[\Theta_H f,f] = \tilde c_{N,\lambda} \int_{\R^{N-1}} J_{\lambda,\xi'}[g_{\xi'}] \,d\xi' \,,
\end{align*}
with $J_{\lambda,\xi'}$ as in the proof of Lemma \ref{repr}. By letting $f$ approach a function of the form $e^{i\xi'\cdot x'}\phi(x_N)$ we see that $I_\lambda[\Theta_H f,f]$ can only be positive (or negative) semi-definite if $J_{\lambda,\xi'}[\phi]$ is so for any $\xi'$. This is equivalent to the kernel $k_{\lambda,\xi'}$ of $J_{\lambda,\xi'}$ being the Laplace transform of a non-negative (or non-positive) measure supported on $[0,\infty)$; see \cite[Prop. 3.2]{FrIsLiSi} for a discrete version of this equivalence assertion. But for $0<\lambda<N-2$,
$$
\frac{d}{dt} k_{\lambda,\xi'}(0) = i \int_\R \frac{\xi_N}{(|\xi'|^2+\xi_N^2)^{(N-\lambda)/2}} \,d\xi_N = 0 \,.
$$
On the other hand, if $\mu$ is a non-negative measure supported on $[0,\infty)$ with $\mu\neq\alpha\delta$ for all $\alpha\geq0$, then $\frac{d}{dt}|_{t=0} \int_0^\infty e^{-st} \,d\mu(s) = -\int_0^\infty s\,d\mu(s)<0$, proving the claim.
\end{remark}

%%%%%%%%%%%%%%%%%%%%%%%%%%%%%%%%%%%%%%%%%%%%%%%%%%%%%%%%%%%%%%%%%%%%%%%%%%5

\subsection{Reduction to the case of half-spaces}
\label{sec:cayley}

Let $B=\{x\in\R^N:\ |x|^2<1\}$ be the unit ball and $e:=(0,\ldots,0,-1)$ (for $N=1$, $e:=-1$). Following \cite{CaLo2} we consider the map $\mathcal B: \R^N\setminus\{e\}\to \R^N$,
$$
\mathcal B (x) := \left( \frac{2x'}{|x-e|^2}, \frac{1-|x|^2}{|x-e|^2} \right) \,.
$$
(For $N=1$, $\mathcal B (x) := (1-|x|^2)/|x-e|^2= (1-x)/(1+x)$.) We note that $\mathcal B$ maps $\R^N\setminus\{e\}$ onto itself and satisfies $\mathcal B^{-1}=\mathcal B$. Moreover, $\mathcal B$ maps $B$ onto the half-space $H := \{ x\in\R^N: x_N>0\}$ and $\R^N\setminus\overline B$ onto $\R^N\setminus\overline H$. Given a function $f$ on $\R^N$ we define
$$
\mathcal B f(x) := \left(\frac{\sqrt 2}{|x-e|} \right)^{2N-\lambda} f(\mathcal B(x)) \,.
$$
The importance of $\mathcal B$ is that it turns inversions through $\partial B$ into reflections on $\partial H$, and that it leaves our energy functional invariant. The precise statement is given in

\begin{lemma}\label{cayley}
 For any function $f$ one has $\mathcal B \Theta_B f = \Theta_H \mathcal B f$. Moreover, $I_\lambda[f]=I_\lambda[\mathcal B f]$.
\end{lemma}

\begin{proof}
The first statement follows by explicit calculation, using, in particular, that $|\mathcal B(x)| = |x+e|/|x-e|$. One way to see the second statement is to note that $\mathcal B(x) = \tau \Theta_{\tilde B} \tau^{-1}(x)$, where $\tau(x)=x+e$ and $\tilde B$ is the ball centered at the origin with radius $\sqrt 2$. Hence, if $\tau f(x) := f(\tau^{-1}(x))=f(x-e)$, then $ \mathcal Bf(x) = \tau \Theta_{\tilde B} \tau^{-1} f(x)$, and the invariance of $I_\lambda$ under $\mathcal B$ follows from its invariance under $\tau$ and $\Theta_{\tilde B}$.
\end{proof}

%%%%%%%%%%%%%%%%%%%%%%%%%%%%%%%%%%%%%%%%%%%%%%%%%%%%%%%%%%%%

\subsection{Proof of Theorem \ref{infrefpos}}
\label{sec:infrefposproof}

We begin by considering the case of a half-space, which after a translation and a rotation we may assume to be $H=\{x : \ x_N>0\}$. A simple calculation shows that
$$
\frac12 \left(I_\lambda[f^i]+I_\lambda[f^o]\right) - I_\lambda[f]
= \int_H \int_H \frac{\overline{(f(x)-f(x',-x_N))}\ (f(y)-f(y',-y_N))}{(|x'-y'|^2+(x_N+y_N)^2)^{\lambda/2}} \,dx\,dy
$$
Defining $g:=f-\Theta_H f$ in $H$ and $g:=0$ in $\R^N\setminus H$, the right side can be rewritten as $I_\lambda[\Theta_H g,g]$. According to Lemma \ref{repr} this is non-negative.

Now assume that $I_\lambda[\Theta_H g,g]=0$ and $\lambda>N-2$. We are going to prove that this implies $g\equiv 0$, which is the same as $f\equiv \Theta_H f$. For $\xi'\in\R^{N-1}$ and $t\geq 0$ let
$$
G_{\xi'}(t) :=(2\pi)^{-(N-1)/2} \int_{\R^{N-1}} e^{-i\xi'x'} g(x',t) \,dx' \,.
$$
By \eqref{eq:repr} and \eqref{eq:cauchy}, for a.e. $\xi'\in\R^{N-1}$ one has
\begin{equation}
 \label{eq:vanish}
\int_0^\infty e^{-\tau t} G_{\xi'}(t) \,dt = 0
\qquad\text{for a.e.}\ \tau\in[|\xi'|,\infty) \,.
\end{equation}
Moreover, by the Minkowski and the Hausdorff-Young inequalities with $p=2N/(2N-\lambda)$
\begin{align*}
\left( \int_{\R^{N-1}} \left( \int_0^\infty |G_{\xi'}(t)|^p \,dt \right)^{p'/p} d\xi' \right)^{p/p'}
& \leq \int_0^\infty \left( \int_{\R^{N-1}} |G_{\xi'}(t)|^{p'} \,d\xi' \right)^{p/p'} dt \\
& \leq c_{N,p} \int_0^\infty \int_{\R^{N-1}} |g(x',t)|^{p} \,dx' dt <\infty \,,
\end{align*}
hence, in particular, $G_{\xi'}\in L^p(\R_+)$ for a.e. $\xi'$. Equality \eqref{eq:vanish} means that for a.e. $\xi'$ the Laplace transform of the function $e^{-t|\xi'|} G_{\xi'}$ vanishes a.e. Hence $G_{\xi'}\equiv 0$ for a.e. $\xi'$ and, by the uniqueness of the Fourier transform, $g\equiv 0$, as claimed.

In order to prove the assertion for balls we may after a translation and a dilation assume that $B=\{ x: |x|<1\}$. Let $f\in L^{2N/(2N-\lambda)}(\R^N)$ and define $f^i$ and $f^o$ as in Theorem \ref{infrefpos} with respect to the ball $B$. Moreover, let $g:=\mathcal B f$ as in Subsection \ref{sec:cayley} and define $g^i$ and $g^o$ as in Theorem \ref{infrefpos} with respect to the half-space $H$. Then by the first part of Lemma \ref{cayley}, $g^i=\mathcal B f^i$ and $g^o=\mathcal B f^o$. Moreover, by the half-space part of Theorem \ref{infrefpos} and the second part of Lemma \ref{cayley},
$$
\frac12 (I_\lambda[f^i]+I_\lambda[f^o]) = \frac12 (I_\lambda[g^i]+I_\lambda[g^o])
\geq I_\lambda[g] = I_\lambda[f] \,.
$$
Moreover, if $\lambda>N-2$ the inequality is strict unless $g=\Theta_H g$, i.e., $f=\Theta_B f$. This completes the proof of Theorem \ref{infrefpos}.

%%%%%%%%%%%%%%%%%%%%%%%%%%%%%%%%%%%%%%%%%%%%%%%%%%%%%%%%%%%%%%%%%%%%%%%

\section{The Li-Zhu lemma}

Our goal in this section is to prove Theorem \ref{yyl}.

\subsection{Preliminary remarks}

Multiplying $\mu$ by a constant if necessary and excluding the trivial case $\mu\equiv 0$, we may and will assume henceforth that $\mu(\R^N)=1$. We begin by noting two easy consequence of the assumption on $\mu$, namely
\begin{equation}
 \label{eq:nopoint}
\mu(\{a\})=0
\qquad
\text{for all}\ a\in\R^N
\end{equation}
and
\begin{equation}
 \label{eq:positive}
\mu(\Omega)>0
\qquad
\text{for any non-empty open}\ \Omega\subset\R^N \,.
\end{equation}
Indeed, for $a\in\R^N$ let $B=\{x:\ |x-a|<r\}$ be the ball from assumption (A). Then $\Theta_B^{-1}(\R^N)=\R^N\setminus\{a\}$ and therefore by \eqref{eq:yyl} $\mu(\R^N\setminus\{a\})=\mu(\R^N)$, proving \eqref{eq:nopoint}. In order to prove \eqref{eq:positive} assume to the contrary that $\mu(\{x:\ |x-a|<\rho\})=0$ for some $a\in\R^N$ and $\rho>0$. If $r$ is the same radius as before, then again by \eqref{eq:yyl} $\mu(\{x:\ |x-a|>r^2/\rho \})=0$. Now let $\tilde a\neq a$ and $\tilde B$ the ball from assumption (A) corresponding to $\tilde a$. There exists an $\epsilon>0$ such that $\Theta_{\tilde B}(\{x:\ 0< |x-\tilde a|<\epsilon\})\subset \{x:\ |x-a|>r^2/\rho \}$. Hence by \eqref{eq:yyl} $\mu(\{x:\ |x-\tilde a|<\epsilon\}) \leq \mu(\{x:\ |x-a|>r^2/\rho \})=0$. Since $\tilde a$ is arbitrary, this proves that $\mu\equiv 0$, contradicting our assumption $\mu(\R^N)=1$.

\bigskip

In the following we call an open ball $B$ a \emph{hemi-ball} (for the measure $\mu$) if $\mu(B)=\mu(\R^N\setminus\overline B)$. Similarly, we call an open half-space $H$ a \emph{hemi-space} (for the measure $\mu$) if $\mu(H)=\mu(\R^N\setminus\overline H)$.
It follows from assumption (A) and \eqref{eq:positive} that for any $a\in\R^N$ there exists a unique hemi-ball centered at $a$, and for any $e\in\Sph^{N-1}$ there exists a unique hemi-space with interior unit normal $e$.

\begin{lemma}
 \label{exballs}
Let $\mu$ be a probability measure satisfying assumption \emph{(A)}. Let $e\in\Sph^{N-1}$ and assume that $\mu(\{ x:\ x\cdot e>0\})=\mu(\{x:\ x\cdot e<0\})$ and that $\mu(\partial B)=0$ for any ball with center $\alpha e$, $\alpha\in\R\setminus\{0\}$. Then for any $u> 0$ there exists a unique hemiball $B$ with $ue\in\partial B$ and with center on $\{\alpha\, e : \alpha\in\R \}$. Moreover, the radius of this ball depends continuously on $u$.
\end{lemma}

\begin{proof}
In order to avoid a technical difficulty we consider first the case where
\begin{equation}
\label{eq:centercont}
\mu(\partial B_u(0))=0 \,.
\end{equation}
Here and in the following, $B_r(a) := \{x:\ |x-a|<r \}$. By \eqref{eq:positive} the function $\rho\mapsto\mu(B_\rho((u-\rho)e))$ increases strictly from $0$ to $\mu(\{x:\ x\cdot e<u\})>1/2$. Moreover, it is continuous since by \eqref{eq:centercont} and the assumption $\mu(\partial B)=0$ for \emph{any} ball with center $\alpha e$. Hence there exists a unique $\rho=\rho_u$ such that $\mu(B_\rho((u-\rho)e))=1/2$. The same argument works if \eqref{eq:centercont} is not satisfied but
\begin{equation}
 \label{eq:centerirrel}
\text{either}\ \mu(B_u(0))>1/2
\qquad
\text{or}\ \mu(\overline{B_u(0)})<1/2 \,.
\end{equation}
If neither \eqref{eq:centercont} nor \eqref{eq:centerirrel} is satisfied, that is, if
\begin{equation*}
 \label{eq:exceptional}
\mu(B_u(0))\leq 1/2 \leq \mu(\overline{B_u(0)})
\qquad\text{and}\qquad 
\mu(\partial B_u(0))>0 \,,
\end{equation*}
(which by \eqref{eq:nopoint} can only happen if $N\geq2$), then $B_u(0)$ coincides with the unique hemi-ball centered at $0$, and we put $\rho_u:=u$.

Assume that $\rho_{u'}$ were not left-continuous at $u'=u$. (The case of right-continuity is similar and hence omitted.) Then there is a sequence $0\leq u_j\leq u_{j+1} <u$ with $u_j\to u$ such that $\rho_j:=\rho_{u_j}$ does not converge to $\rho_u$. We abbreviate $B_j:= B_{\rho_{j}}((u_{j}-\rho_{j})e)$ and note that by \eqref{eq:positive}
\begin{equation*}
u_j-2\rho_j\leq u_{j+1}-2\rho_{j+1} \leq u- 2\rho_u \,.
\end{equation*}
Hence the $\rho_{j}$ converge to some $\rho_*:=\lim \rho_j \geq \rho_u$. Since the $\rho_j$ do not converge this inequality is strict, which implies $\overline{B_{\rho_u}((u-\rho_u)e)}\setminus\{ue\} \subset B_{\rho_*}((u-\rho_*)e)=:B_*$ and hence by \eqref{eq:nopoint} and \eqref{eq:positive}
\begin{equation}
 \label{eq:contra}
\mu(\overline{B_{\rho_u}((u-\rho_u)e)}) < \mu(B_*) \,.
\end{equation}
On the other hand, $\chi_{B_j}(x)\to 1$ if $x\in B_*$ and $\chi_{B_j}(x)\to 0$ if $x\in \R^N\setminus\overline{B_*}$. Since $\mu(\partial B_*)=0$, dominated convergence implies that $\mu(B_*)= \lim \mu(B_j)=1/2$. This contradicts \eqref{eq:contra} and the fact that $B_{\rho_u}((u-\rho_u)e)$ is a hemi-ball.
\end{proof}

\begin{corollary}
 \label{exballscor}
Let $\mu$ and $e$ be as in Lemma \ref{exballs}. If $0\leq s<t$, then there exists a hemi-ball $B$ such that $\Theta_{B}(se)=te$.
\end{corollary}

\begin{proof}
First, we assume that $s>0$. For any $u\in [s,t]$ let $B_u$ be the ball constructed in Lemma \ref{exballs} which passes through $u e$, and let $\rho_u$ and $a_u$ be its radius and center. We want to determine $u$ such that $\Theta_{B_u}(te)=se$, which is equivalent to having $f(u):= |t e -a_u||s e-a_u| - \rho_u^2=0$. Since $\mu(\{ x:\ x\cdot e>0\})=\mu(\{x:\ x\cdot e<0\})$ we have $a_u=(u-\rho_u)e$. Moreover, $f(s)= \rho_s (t-s)>0$.

Since $f$ is continous by Lemma \ref{exballs}, the assertion will follow if we can find a $u\in [s,t]$ with $f(u)<0$. By continuity, one has $u-\rho_u\leq s$ if $u-s$ is small. We distinguish two cases according to whether $u-\rho_u\leq s$ for all $u\in[s,t]$ or not. In the first case, one has $t-\rho_t\leq s$ and hence $f(t)=\rho_t (s-t)<0$. Otherwise, one has $u-\rho_u= s$ for some $u\in[s,t]$ and then $f(u)=-\rho_u^2<0$. This completes the proof for $s>0$.

The proof for $s=0$ is similar. One easily checks that $\rho_u\to-\infty$ and $|a_u|\to\infty$ as $u\to 0$, which implies that $f(u)\to +\infty$ as $u\to 0$. The assertion follows as before from $f(u)<0$ for some $u\in(0,t]$.
\end{proof}

Our next result concerns arbitrary measures \emph{without} requiring assumption (A).

\begin{lemma}\label{ac}
 Let $\mu$ be a non-negative finite Borel measure on $\R^N$ with $\mu(\{0\})=0$. Assume that $\mu$ is radial in the sense that if $B$ and $B'$ are balls with the same radius and with centers $a$ and $a'$ satisfying $|a|=|a'|$, then $\mu(B)=\mu(B')$. Moreover, assume that $\mu$ is decreasing, in the sense that if $B$ and $B'$ are open balls with the same radius $r$ and with centers $a=te$ and $a'=t'e$ satisfying $t\geq t'\geq 0$, $t-r>t'+r$ and $e\in\Sph^{N-1}$, then $\mu(B)\leq \mu(B')$. Then $\mu$ is absolutely continuous.
\end{lemma}

\begin{proof}
 We shall make use of two facts. First, if $\mu$ has a non-zero singular part on $\{x:\ |x|\geq R\}$ for some $R\geq 0$, then
\begin{equation}
 \label{eq:sing}
\limsup_{r\to0}\sup_{|a|\geq R} \frac{\mu(\{x:\ |x-a|<r \})}{|\{x:\ |x-a|<r \}|} =\infty \,.
\end{equation}
(Here is a short proof: Assume \eqref{eq:sing} were wrong, then there were $\delta, M>0$ such that $\mu(B)\leq M |B|$ for any ball $B$ of radius $\leq\delta$ and center at a distance $\geq R$ from the origin. Since $\mu$ is singular there exists a Borel set $A\subset \{x:\ |x|\geq R\}$ with $0<\mu(A)<\infty$ and $|A|=0$. Choose $\epsilon<M/\mu(A)$. By regularity of the Lebesgue measure there exists an open set $U\subset\R^N$ such that $A\subset U$ and $|U|\leq \epsilon$. By the Besicovitch covering lemma \cite[Sec. 1.5.2, Cor. 2]{EvGa} there exist countably many disjoint balls $B_j$ of radii $r_j\leq\delta$ and with centers in $A$ such that $\bigcup_j B_j\subset U$ and $\mu\left(A-\bigcup_j B_j\right)=0$. Hence
$$
0<\mu(A)=\sum_j \mu(A\cap B_j) \leq \sum_j \mu(B_j) \leq M \sum_j |B_j| \leq M |U| \leq M\epsilon \,.
$$
This contradicts the choice of $\epsilon$ and hence proves \eqref{eq:sing}.)

The second fact we use is an elementary, qualitative version of the sphere packing theorem: there are constants $c_1>0$ and $1>c_2>0$ (depending only on $N$) such that for any $R>0$, the number of disjoint balls of radius $r$ within a ball of radius $R$ is bounded from below by $c_1(R/r)^N$ provided $r\leq c_2 R$.

Using these two fact we are now going to prove that a radial, decreasing measure $\mu$ with $\mu(\{0\})=0$ is absolutely continuous. Suppose not, then $\mu$ has a non-zero singular part on $\{x:\ |x|\geq R\}$ for some $R>0$. Choose
$$
M > c_1^{-1}(1+c_2)^{N} \, \frac{\mu(\{x:\ |x|<R\})}{|\{x:\ |x|<R\}|} \,.
$$
According to \eqref{eq:sing} there is a ball $B$ with center $a$ and radius $r$ such that $|a|\geq R$, $r \leq c_2(1+c_2)^{-1} R$ and $\mu(B)\geq M |B|$. By the second fact mentioned above, the ball $\{x:\ |x|<R/(1+c_2)\}$ contains disjoint balls $B_1,\ldots B_n$ of the same radius $r$ with $n\geq c_1 (1+c_2)^{-N} (R/r)^N$. Since $\mu$ is radial and decreasing one has $\mu(B_j)\geq\mu(B)$ for any $j$ and hence
$$
\mu(\{x:\ |x|<R \}) \geq \sum_j \mu(B_j) \geq n \mu(B) \geq c_1 (1+c_2)^{-N} R^{N} (|B|/r^N) M \,.
$$
Recalling the choice of $M$, we arrive at a contradiction.
\end{proof}

%%%%%%%%%%%%%%%%%%%%%%%%%%%%%%%%%%%%%%%%%%%%%%%%%%%%%%%%%%%%

\subsection{Proof of Theorem \ref{yyl}}

Taking $e_1,\ldots,e_N\in\Sph^{N-1}$ the canonical basis in $\R^N$, assumption (A) gives us hemi-spaces $H_j$, $j=1,\ldots,N$, with interior unit normal $e_j$. After a translation we may and will assume that the $H_j$ intersect at the origin.

\emph{Step 1.} We claim that $\mu$ is radial, in the sense that if $B$ and $B'$ are balls with the same radius and with centers $a$ and $a'$ satisfying $|a|=|a'|$, then $\mu(B)=\mu(B')$.

There is an $e\in\Sph^{N-1}$ such that $B'=\Theta_H^{-1}(B)$ for $H=\{x:\ x\cdot e>0\}$. Hence the assertion will follow if we can prove that $H$ is the hemi-space corresponding to $e$ according to assumption (A). Because of the uniqueness of hemi-spaces, we only need to prove that $\mu(H)=\mu(\R^N\setminus\overline{H})$. But this equality follows from \eqref{eq:yyl} and the fact that $\R^N\setminus\overline{H}=-H=\Theta_{H_1}^{-1}(\cdots(\Theta_{H_N}^{-1}(H)))$.

\emph{Step 2.} We claim that $\mu(\partial B)=0$ for any ball $B$ centered away from the origin.

Because of \eqref{eq:nopoint} we only need to consider $N\geq 2$. Assume to the contrary that $\mu(\partial B)=\epsilon>0$ for a ball of radius $r$ centered at $a\neq 0$. Let $n>\epsilon^{-1}$. There exist balls $B_1,\ldots,B_n$ with the same radius $r$, with centers $a_j$ satisfying $|a_j|=|a|$ and with the property that $\partial B_i\cap \partial B_j$ is a set of codimension $\leq 2$ for $i\neq j$. By the same argument as in Step 1, $\mu(\partial B_j)=\mu(\partial B)=\epsilon$ for all $j$. Hence
$$
\mu(\bigcup_j \partial B_j) \geq \sum_j \mu(\partial B_j) - \sum_{i< j} \mu((\partial B_i)\cap (\partial B_j))
= n\epsilon - \sum_{i< j} \mu((\partial B_i)\cap (\partial B_j)) \,.
$$
Since $n>\epsilon^{-1}$ this will contradict $\mu(\R^N)=1$, provided we can prove that $\mu((\partial B_i)\cap (\partial B_j))=0$ for all $i\neq j$. Now we iterate the argument taking rotated copies of $(\partial B_i)\cap (\partial B_j)$ which intersect in sets of codimension $\leq 3$. After finitely many iterations the sets will only intersect in points, which according to \eqref{eq:nopoint} have measure zero. 

\emph{Step 3.} We claim that $\mu$ is decreasing, in the sense that if $B$ and $B'$ are open balls with the same radius $r$ and with centers $a=te$ and $a'=t'e$ satisfying $t\geq t'\geq 0$, $t-r>t'+r$ and $e\in\Sph^{N-1}$, then $\mu(B)\leq \mu(B')$.

Indeed, according to Corollary \ref{exballscor} there is a hemi-ball $\tilde B$ with center on the $e$-axis such that $\Theta_{\tilde B} ((t'+r)e) = (t-r)e$. A short calculation shows that $B' \supset \Theta_{\tilde B}^{-1}(B)$, and hence by the assumption $\mu(B')\geq \mu(\Theta_{\tilde B}^{-1}(B))=\mu(B)$.

\emph{Step 4.} According to Steps 1 and 3 and Lemma \ref{ac}, $d\mu(x) = v(x)\,dx$ where $v$ is a symmetric decreasing function in $L^1(\R^N)$. Moreover, the assumption on $\mu$ implies that for any $a\in\R^N$ there exists an $r_a>0$ and a set of full measure such that \eqref{eq:pweq} holds for any $x$ from this set. We claim that $v$ is continuous.

Since $v$ is symmetric decreasing we may assume that it is lower semi-continuous. In order to prove continuity we only need to show that the radial limits from inside and outside at any point $x$ coincide. We first assume that $x\neq 0$ and write $x=t e$ with $t>0$ and $e\in\Sph^{N-1}$. Let $B$ be the hemi-ball constructed in Lemma \ref{exballs} and denote its center and radius by $a$ and $r_a$. Choose a sequence $t_j>t$ with $t_j\to t$ such that \eqref{eq:pweq} holds with $x$ replaced by $t_j e$. Define $\tilde t_j<t$ such that $\tilde t_j e$ is the inversion of $t_j e$ with respect to $B$. Since $|t_je-a|\to|x-a|=r_a$, \eqref{eq:pweq} implies that $\lim v(t_je) = \lim v(\tilde t_je)$, that is $v$ is continuous at $x$.

In order to prove that $v$ is continuous at the origin, we fix some $a\neq 0$ and let $x\to\infty$ in \eqref{eq:pweq}. Using the continuity of $v$ at $a$ which we have just shown, we find that $\lim_{|x|\to\infty} |x|^{2N} v(x)$ exists and equals $r_a^{2N} v(a)$, which is clearly finite. Now we can use \eqref{eq:pweq} with $a=0$ to conclude that $v(x) = (r_0/|x|)^{2N} v( r_0^2 x/|x|^2 ) \to r_0^{-2N} r_a^{2N} v(a)$ as $|x|\to 0$, that is, $v$ is continuous at the origin.

\emph{Step 5.} We claim that $v$ is differentiable.

First, let $0\neq x=t e$ with $t>0$ and $e\in\Sph^{N-1}$ and let $B$ be as in the previous step. We are going to show that $\partial_r v(x) = -Nv(x)/|x-a|$. For the sake of definiteness we show this for the derivative from the outside. Let $t_j>t$ with $t_j\to t$. According to Corollary \ref{exballscor} there are hemi-balls $B_j$ with centers $a_j$ and radii $r_j$ such that $\Theta_{B_j} x = t_{j+1} e$. By Lemma \ref{exballs}, $\rho_j\to r_a$ and hence $a_j\to a$. Hence by \eqref{eq:pweq} (which holds for \emph{every} $x$ by Step 5) $v(t e_j) = \gamma_j^N v(x)$ with $\gamma_j:= \rho_j^2/|t_j e-a_j|^2$. Using that
$$
|t_j e-x| = |t_j e- a_j|-|x-a_j| = |t_j e-a_j| \left( 1- \gamma_j \right)
$$
we have
$$
\frac{v(t_j e)-v(x)}{|t_j e -x|}
= - \frac{1-\gamma_j^N }{1-\gamma_j }\ \frac{v(x)}{|t_j e-a_j|} \,.
$$
Letting $j\to\infty$ and noting that $(1-\gamma_j^N)/(1-\gamma_j)\to N$ we obtain the claimed formula.

For $x=0$ one has $\partial_r v(0)=0$. This is shown by a similar argument, but now $\rho_j\to \infty$ and $|a_j|\to\infty$ with $\rho_j/|a_j|\to 1$.

\emph{Step 6.} Following \cite{LZh} we conclude that $v$ has the form claimed in Theorem \ref{yyl}. Indeed, for any fixed $a\in\R^N$ a Taylor expansion shows that as $|x|\to\infty$
\begin{align*}
& \left(\frac{r_a}{|x-a|}\right)^{2N} v\left( \frac{r_a^2 (x-a)}{|x-a|^2}+a \right) \\
& \qquad = \left(\frac{r_a}{|x|}\right)^{2N} \left( v(a) + \frac{a\cdot x}{|x|^2} \left( \frac{r_a^2 \partial_r v(a)}{|a|} + 2N v(a) \right) +o(|x|^{-1}) \right) \,.
\end{align*}
In particular, at $a=0$ where $\partial_r v(0)=0$,
$$
\left(\frac{r_0}{|x|}\right)^{2N} v\left( \frac{r_0^2 x}{|x|^2}\right)
= \left(\frac{r_0}{|x|}\right)^{2N} \left( v(0) + o(|x|^{-1}) \right) \,.
$$
Because of \eqref{eq:pweq} both expansions coincide and we infer that $r_a^{2N} v(a) = r_0^{2N} v(0)$ (which we already know from Step 5) and that $r_a^2 \partial_r v(a) + 2N |a| v(a) = 0$. Hence
$$
\partial_r \ v^{-1/N}(a) = \frac{2\ |a|}{r_0^{2} \ v(0)^{1/N}} \,,
$$
and the solution to this ordinary differential equation is $v(a)=r_0^{2N} v(0)(r_0^2+|a|^2)^{-N}$. An easy calculation shows that these functions indeed satisfy assumption (A). This completes the proof of Theorem \ref{yyl}.
\lanbox

%%%%%%%%%%%%%%%%%%%%%%%%%%%%%%%%%%%%%%%%%%%%%%%%%%%%%%%%%%%%%%%%%%%%%%%%%%%%%%%%%%%%%%%%%%%

\subsection*{Note added in proof}

Recently we managed to apply the methods presented in this paper to obtain the sharp logarithmic HLS inequality, that is, the analogue of \eqref{eq:hls1} with $|x-y|^{-\lambda}$ replaced by $\log |x-y|$, in dimensions $N=1$ and $2$. This result is originally due to Carlen and Loss \cite{CaLo3} and Beckner \cite{Be}. Details will appear in \cite{FrLi} (where we also present our original proof of Theorem \ref{infrefpos} using Gegenbauer polynomials).

%%%%%%%%%%%%%%%%%%%%%%%%%%%%%%%%%%%%%%%%%%%%%%%%%%%%%%%%%%%%%%%%%%%%%%%%%%%%%%%%%%%%%%%%%%%%

\bibliographystyle{amsalpha}

\end{document}